\documentclass[psamsfonts]{proc-l}

\usepackage{amscd, amsmath, amsfonts, amssymb, latexsym}
\usepackage[mathscr]{eucal}
\copyrightinfo{2012}{American Mathematical Society}

\newtheorem{theorem}{Theorem}
\newtheorem{lemma}{Lemma}

\title[Landen inequalities for special functions]{Landen inequalities for special functions}

\author[\'Arp\'ad Baricz]{\'Arp\'ad Baricz}

\address{Department of Economics, Babe\c{s}-Bolyai University,
400591 Cluj-Napoca, Romania} \email{bariczocsi@yahoo.com}



\begin{document}

\begin{abstract}
 In this paper our aim is to present some Landen inequalities for Gaussian hypergeometric functions, confluent hypergeometric functions, generalized Bessel functions and for general power series. Our main results complement and generalize some known results in the literature.
\end{abstract}

\maketitle

\section{Introduction}
\setcounter{equation}{0}

Let us consider the Gaussian hypergeometric function $F(a,b;c;\cdot):(-1,1)\to\mathbb{R},$ which for real numbers $a,$ $b$ and
$c$ such that $c$ is not in $\{0,-1,\dots\}$ has the infinite series representation
$$F(a,b;c;x):={}_2F_1(a,b;c;x)=\sum_{n\geq
0}\frac{{(a)}_n{(b)}_n}{{(c)}_n}\cdot\frac{x^n}{n!},$$
where $x\in(-1,1),$ ${(a)}_0=1$ for $a\neq 0$ and
${(a)}_n=a(a+1)\dots(a+n-1)=\Gamma(a+n)/\Gamma(a)$ for each
$n\in\{1,2,\dots\}$ denotes the Pochhammer (or Appell) symbol. Some of the
most important properties of the complete elliptic integral of the
first kind, i.e. $\mathcal{K}(r),$ defined by
$$\mathcal{K}(r)=\frac{\pi}{2}F\left(\frac{1}{2},\frac{1}{2};1;r^2\right)
=\int_{0}^{\frac{\pi}{2}}{(1-r^2\sin^2t)}^{-\frac{1}{2}}dt,\ \ \ r\in(0,1),$$ are the
Landen identities proved in 1771, \cite{alm}:
\begin{equation}\label{landenid} \mathcal{K}\left(\frac{2\sqrt{r}}{1+r}\right)=(1+r)\mathcal{K}(r), \ \ \ \ \
\mathcal{K}\left(\frac{1-r}{1+r}\right)=\frac{1+r}{2}\mathcal{K}\left(\sqrt{1-r^2}\right)\end{equation}
These Landen identities, which are in fact equivalent to each other, have been the starting points of the investigations of Qiu and Vuorinen \cite{qiu}, and recently of Simi\'c and Vuorinen \cite{simic}. In this paper, motivated by \cite{simic}, we make a contribution to the subject by showing that \cite[Theorem 2.1]{simic}, proved for the zero-balanced hypergeometric function $F(a,b;a+b;\cdot)$, can be extended to the hypergeometric function $F(a,b;c;\cdot)$ and also to general power series. Moreover, we prove that, by using a generalization of the first Landen identity in \eqref{landenid}, the Landen inequalities for the Gaussian hypergeometric functions can be improved in some cases. Our main results complement the results from \cite{bariczcmft,bariczbook,qiu} and \cite{simic}.

\section{Landen inequalities for power series}

Let us recall a result of Biernacki and Krzy\.z \cite{biernacki}, which we will use in the sequel.

\begin{lemma}\label{lemmapower}
Consider the power series $f(x)=\sum\limits_{n\geq0}a_nx^n$
and $g(x)=\sum\limits_{n\geq0}b_nx^n,$ where $a_n\in\mathbb{R}$ and $b_n>0$ for all $n\in\{0,1,\dots\},$ and suppose that both converge on $(-r,r),$ $r>0.$
If the sequence $\{a_n/b_n\}_{n\geq 0}$ is increasing (decreasing),
then the function $x\mapsto f(x)/g(x)$ is increasing (decreasing) too on $(0,r).$
\end{lemma}

For different proofs and various applications of this result the
interested reader is referred to the papers
\cite{anderson,alzer,bala,bariczcmft,bariczedin,bariczbook,heikkala,ponnusamy,simic}
and to the references therein.

Our first main result is the following theorem.

\begin{theorem}\label{landen}
Let $a,b,c\in\mathbb{R}$ such that $c$ is not a negative integer or zero and consider the function $Q:(0,1)\to(0,\infty),$ defined by $Q(x)=F(a,b;c;x)/F\left(\frac{1}{2},\frac{1}{2};1;x\right).$ The following assertions are true:
\begin{enumerate}
\item[\bf a.] If $a+b\geq c$ and $4ab\geq\max\{1,c\},$ then $Q$ is increasing, and consequently
\begin{equation}\label{ineq1}
F\left(a,b;c;\frac{4r}{(1+r)^2}\right)\geq (1+r)\cdot F(a,b;c;r^2),
\end{equation}
\begin{equation}\label{ineq2}
F\left(a,b;c;\left(\frac{1-r}{1+r}\right)^2\right)\leq\frac{1+r}{2}\cdot F\left(a,b;c;1-r^2\right)
\end{equation}
hold for each $r\in(0,1).$
\item[\bf b.] If $a+b\leq c$ and $4ab\leq\min\{1,c\},$ then $Q$ is decreasing, and consequently
\begin{equation}\label{ineq3}
F\left(a,b;c;\frac{4r}{(1+r)^2}\right)\leq (1+r)\cdot F(a,b;c;r^2),
\end{equation}
\begin{equation}\label{ineq4}
F\left(a,b;c;\left(\frac{1-r}{1+r}\right)^2\right)\geq\frac{1+r}{2}\cdot F\left(a,b;c;1-r^2\right)
\end{equation}
hold for each $r\in(0,1).$
\end{enumerate}
\end{theorem}

\begin{proof}[\bf Proof]
{\bf a.} \& {\bf b.} We shall apply Lemma \ref{lemmapower}. Since $Q(x)$ can be rewritten as
$$Q(x)=\frac{F(a,b;c;x)}{F\left(\frac{1}{2},\frac{1}{2};1;x\right)}=\frac{\displaystyle\sum_{n\geq
0}\frac{{(a)}_n{(b)}_n}{{(c)}_n}\cdot\frac{x^n}{n!}}{\displaystyle\sum_{n\geq
0}\frac{{\left(\frac{1}{2}\right)}_n{\left(\frac{1}{2}\right)}_n}{{(1)}_n}\cdot\frac{x^n}{n!}},$$
in view of Lemma \ref{lemmapower}, the monotonicity of the quotient $Q$ depends on the monotonicity of the quotient sequence $\{\alpha_n\}_{n\geq0},$ defined by
$$\alpha_n=\frac{{(a)}_n{(b)}_n}{{(c)}_n}\cdot\frac{{(1)}_n}{{\left(\frac{1}{2}\right)}_n{\left(\frac{1}{2}\right)}_n}.$$
Now, observe that $$\frac{\alpha_{n+1}}{\alpha_n}=\frac{(n+a)(n+b)(n+1)}{(n+c)\left(n+\frac{1}{2}\right)\left(n+\frac{1}{2}\right)}\geq1$$
if and only if $$\Delta_n=(a+b-c)n^2+\left(a+b-c+ab-\frac{1}{4}\right)n+ab-\frac{c}{4}\geq0.$$ Thus, if $a+b\geq c$ and $4ab\geq\max\{1,c\},$ then $\Delta_n\geq0$ for all $n\in\{0,1,\dots\},$ that is, the sequence $\{\alpha_n\}_{n\geq0}$ is increasing, and consequently by using Lemma \ref{lemmapower} the function $Q$ is increasing. In other words, if $0<x<y<1,$ then we have $Q(x)<Q(y).$ Now, choosing $x=x(r)=r^2$ and $y=y(r)=4r/(1+r)^2,$ we obtain the inequality
$$\frac{F(a,b;c;r^2)}{F\left(\frac{1}{2},\frac{1}{2};1;r^2\right)}\leq\frac{F\left(a,b;c;\frac{4r}{(1+r)^2}\right)}
{F\left(\frac{1}{2},\frac{1}{2};1;\frac{4r}{(1+r)^2}\right)},$$
that is,
$$F(a,b;c;r^2)\leq F\left(a,b;c;\frac{4r}{(1+r)^2}\right)\cdot \frac{\mathcal{K}(r)}{\mathcal{K}\left(\frac{2\sqrt{r}}{1+r}\right)},$$
which in view of the first Landen identity in \eqref{landenid} is equivalent to \eqref{ineq1}. Similarly, by choosing $x=x(r)=[(1-r)/(1+r)]^2$ and $y=y(r)=1-r^2$ we get the inequality
$$\frac{F\left(a,b;c;\left(\frac{1-r}{1+r}\right)^2\right)}{F\left(\frac{1}{2},\frac{1}{2};1;\left(\frac{1-r}{1+r}\right)^2\right)}\leq\frac{F\left(a,b;c;1-r^2\right)}
{F\left(\frac{1}{2},\frac{1}{2};1;1-r^2\right)},$$
that is,
$$F\left(a,b;c;\left(\frac{1-r}{1+r}\right)^2\right)\leq F\left(a,b;c;1-r^2\right)\cdot \frac{\mathcal{K}\left(\frac{1-r}{1+r}\right)}{\mathcal{K}\left(\sqrt{1-r^2}\right)},$$
which in view of the second Landen identity in \eqref{landenid} is equivalent to \eqref{ineq2}. This proves part {\bf a}. The proof of part {\bf b} is similar, and thus we omit the details.
\end{proof}

First of all we mention that the Landen inequalities \eqref{ineq1} and \eqref{ineq2} are equivalent, as well as the inequalities \eqref{ineq3} and \eqref{ineq4}. Namely, if we change $r$ to $(1-r)/(1+r)$ in \eqref{ineq1} and \eqref{ineq3}, then we obtain \eqref{ineq2} and \eqref{ineq4}. Similarly, if we change $(1-r)/(1+r)$ to $r$ in \eqref{ineq2} and \eqref{ineq4}, then we obtain \eqref{ineq1} and \eqref{ineq3}. It should be also mentioned here that in Theorem \ref{landen} it is not necessary to assume that $a,b$ and $c$ are positive numbers. However, if we suppose in inequalities \eqref{ineq1} and \eqref{ineq3} in particular that $a,b>0$ and $c=a+b,$ then we reobtain \cite[Theorem 2.1]{simic}, which was obtained recently by Simi\'c and Vuorinen. We mention that the condition $4ab\geq\max\{1,c\}$ in part {\bf a} of Theorem \ref{landen} reduces to $4ab\geq a+b,$ since by applying the arithmetic mean - geometric mean inequality for the numbers $a$ and $b,$ the above condition implies that $4ab\geq1.$ Similarly, the condition
$4ab\leq\min\{1,c\}$ in part {\bf b} of Theorem \ref{landen} reduces to $4ab\leq 1,$ since by applying the geometric mean - harmonic mean inequality for the numbers $a$ and $b,$ the above condition implies that $4ab\leq a+b.$ We also note that a general result about the monotonicity of quotients of Gaussian hypergeometric functions was given by Ponnusamy and Vuorinen in \cite[Theorem 2.31]{ponnusamy}. Finally, we note that the inequality \eqref{ineq4} was proved earlier by Qiu and Vuorinen \cite[Theorem 1.2]{qiu}, but just for $a,b\in(0,1)$ and $c=a+b\leq1.$ Observe that in this case the condition $4ab\leq 1$ is clearly satisfied since $2\sqrt{ab}\leq a+b\leq1.$

Now, let us consider the sequence $\{\omega_n\}_{n\geq0},$ defined by
$$\omega_n=\left[\frac{{(1)}_n}{{\left(\frac{1}{2}\right)}_n}\right]^2=\pi\cdot \left[\frac{\Gamma(n+1)}{\Gamma\left(n+\frac{1}{2}\right)}\right]^2.$$
By using this sequence we would like to show a generalization of Theorem \ref{landen}. Note that since the proof of this general result go along the lines introduced in the proof of Theorem \ref{landen}, we omit the details. This result complements \cite[Theorem 3.1]{bariczcmft}.

\begin{theorem}\label{power}
Suppose that the power series $f(x)=\sum\limits_{n\geq0}a_nx^n$ is convergent for all $x\in(0,1),$ where $a_n\in\mathbb{R}$ for all $n\in\{0,1,\dots\},$ and assume that the sequence $\{a_n\cdot \omega_n\}_{n\geq0}$ is increasing. Then the function $x\mapsto f(x)/\mathcal{K}(\sqrt{x})$ is increasing on $(0,1),$ and by using the notation $\lambda_f(x)=f(x^2)$ we have the Landen type inequality for all $r\in(0,1)$
\begin{equation}\label{ineq5}
\lambda_f\left(\frac{2\sqrt{r}}{1+r}\right)\geq (1+r)\cdot \lambda_f(r).
\end{equation}
Moreover, if the sequence $\{a_n\cdot \omega_n\}_{n\geq0}$ is decreasing, then $x\mapsto f(x)/\mathcal{K}(\sqrt{x})$ is decreasing on $(0,1),$ and consequently \eqref{ineq5} is reversed.
\end{theorem}

Observe that the sequence $\{\omega\}_{n\geq0}$ is increasing, thus if the sequence $\{a_n\}_{n\geq0}$ is also increasing, then the power series $f(x)$ of Theorem \ref{power} immediately satisfies the Landen type inequality \eqref{ineq5}, which is in fact equivalent to
$$\lambda_f\left(\frac{1-r}{1+r}\right)\leq \frac{1+r}{2}\cdot \lambda_f\left(\sqrt{1-r^2}\right).$$
Note that if we consider, as in \cite{bariczcmft,bariczbook}, the generalized Bessel function $u_{\nu}:(0,\infty)\to\mathbb{R}$ and the Kummer hypergeometric function $\Phi(p,q;\cdot):(0,\infty)\to\mathbb{R},$ defined by
$$u_{\nu}(x)=\sum_{n\geq0}\frac{\left(-\frac{c}{4}\right)^n}{{(\kappa)}_n}\cdot \frac{x^n}{n!}\ \ \mbox{and}\ \ \Phi(p,q;x)=\sum_{n\geq0}\frac{(p)_n}{(q)_n}\cdot \frac{x^n}{n!},$$
where $\nu,b,c,p,q\in\mathbb{R},$ $\kappa=\nu+\frac{b+1}{2}$ and $q$ are not in $\{0,-1,\dots\},$ then it can be shown that the sequences
$$\left\{\frac{\left(-\frac{c}{4}\right)^n}{{(\kappa)}_nn!}\cdot\omega_n\right\}_{n\geq0}\ \ \mbox{and}\ \
\left\{\frac{(p)_n}{{(q)}_nn!}\cdot\omega_n\right\}_{n\geq0}$$ are decreasing if $\kappa\geq\max\left\{0,-c,-\frac{c+1}{4}\right\}$ and $q\geq\max\left\{0,4p,p+\frac{3}{4}\right\}.$ Thus, if use the notations $\lambda_{\nu}(r)=u_{\nu}(r^2)$ and $\lambda_{\Phi}(r)=\Phi(p,q;r^2),$ then in view of Theorem \ref{landen2} we obtain the following result. Note that this result complements \cite[Theorem 2.3]{bariczcmft} and \cite[Corollary 3.2]{bariczcmft}.

\begin{theorem} Let $\nu,b,c,p$ and $q$ be real numbers such that $\kappa\geq\max\left\{-1,-c,-\frac{c+1}{4}\right\}$ and $q\geq\max\left\{0,4p,p+\frac{3}{4}\right\}.$ Then $x\mapsto u_{\nu}(x)/\mathcal{K}(\sqrt{x})$ and $x\mapsto \Phi(p,q;x)/\mathcal{K}(\sqrt{x})$ are decreasing on $(0,1)$ and consequently for all $r\in(0,1)$ we have
$$\lambda_{\nu}\left(\frac{2\sqrt{r}}{1+r}\right)\leq (1+r)\cdot \lambda_{\nu}(r)\ \ \mbox{and}\ \
\lambda_{\Phi}\left(\frac{2\sqrt{r}}{1+r}\right)\leq (1+r)\cdot \lambda_{\Phi}(r).$$
\end{theorem}

Now, let us consider the following hypergeometric transformation \cite[p. 128]{askey}
\begin{equation}\label{transf}F\left(a,b;2b;\frac{4r}{(1+r)^2}\right)=(1+r)^{2a}\cdot F\left(a,a+\frac{1}{2}-b; b+\frac{1}{2};r^2\right),\end{equation}
which can be regarded as the generalization of the first Landen identity in \eqref{landenid}. By using this transformation we can obtain the following result. Observe that if the conditions of part {\bf a} of Theorem \ref{landen2} are valid and in addition $c>0$ and $a\geq \frac{1}{2},$ then the Landen inequality \eqref{ineq6} improves \eqref{ineq1}. Similarly, if the conditions of part {\bf b} of Theorem \ref{landen2} are valid and in addition $a\leq \frac{1}{2},$ then the Landen inequality \eqref{ineq7} improves \eqref{ineq3}.

\begin{theorem}\label{landen2}
Let $a,b>0$ and $c\in\mathbb{R}$ such that $c$ is not a negative integer or zero. The following assertions are true:
\begin{enumerate}
\item[\bf a.] If $\max\{1,c\}\leq2b\leq a+\frac{1}{2}$ or $c\leq 2b\leq a$ or $3c\leq 6b\leq\min\{6a,4a+1\},$ then
\begin{equation}\label{ineq6}
F\left(a,b;c;\frac{4r}{(1+r)^2}\right)\geq (1+r)^{2a}\cdot F(a,b;c;r^2)
\end{equation}
holds for each $r\in(0,1).$
\item[\bf b.] If $a+\frac{1}{2}\leq2b\leq\min\{1,c\}$ or $\max\{6a,4a+1\}\leq6b\leq3c,$ then
\begin{equation}\label{ineq7}
F\left(a,b;c;\frac{4r}{(1+r)^2}\right)\leq (1+r)^{2a}\cdot F(a,b;c;r^2)
\end{equation}
holds for each $r\in(0,1).$
\end{enumerate}
\end{theorem}

\begin{proof}[\bf Proof]
{\bf a.} \& {\bf b.} We proceed similarly as in the proof of Theorem \ref{landen}. For this, first we consider the function $T:(0,1)\to(0,\infty),$ defined by
$$T(x)=\frac{F(a,b;c;x)}{F\left(a,b;2b;x\right)}=\frac{\displaystyle\sum_{n\geq
0}\frac{{(a)}_n{(b)}_n}{{(c)}_n}\cdot\frac{x^n}{n!}}{\displaystyle\sum_{n\geq
0}\frac{{\left(a\right)}_n{\left(b\right)}_n}{{(2b)}_n}\cdot\frac{x^n}{n!}}.$$
Now, in view of Lemma \ref{lemmapower}, for the monotonicity of the quotient $T$ we need to study the monotonicity of the quotient sequence $\{\beta_n\}_{n\geq0},$ defined by $\beta_n={{(2b)}_n}/{{(c)}_n}.$ Since $\beta_{n+1}/\beta_n=(n+2b)/(n+c),$ it is clear that the sequence $\{\beta_n\}_{n\geq0}$ is increasing (decreasing) if $2b\geq c$ ($2b\leq c$). Now, if we consider the case $2b\geq c,$ then $\{\beta_n\}_{n\geq0}$ is increasing, and applying Lemma \ref{lemmapower} the function $T$ is increasing. In other words, if $0<x<y<1,$ then we have $T(x)<T(y).$ Thus, choosing $x=x(r)=r^2$ and $y=y(r)=4r/(1+r)^2,$ we obtain the inequality
$$\frac{F(a,b;c;r^2)}{F\left(a,b;2b;r^2\right)}\leq\frac{F\left(a,b;c;\frac{4r}{(1+r)^2}\right)}
{F\left(a,b;2b;\frac{4r}{(1+r)^2}\right)},$$
which in view of \eqref{transf} is equivalent to
\begin{equation}\label{ineq8}F(a,b;c;r^2)\leq F\left(a,b;c;\frac{4r}{(1+r)^2}\right)\cdot \frac{F(a,b;2b;r^2)}{(1+r)^{2a}F\left(a,a+\frac{1}{2}-b;b+\frac{1}{2};r^2\right)}.\end{equation}
Observe that if $2b\leq a+\frac{1}{2}$ and $2b\geq 1,$ then for all $n\in\{0,1,\dots\}$ we have ${(b)}_n\leq {\left(a+\frac{1}{2}-b\right)}_n$ and ${\left(b+\frac{1}{2}\right)}_n\leq {(2b)}_n.$ Similarly, if $0<2b\leq a,$ then we clearly have ${\left(b+\frac{1}{2}\right)}_n\leq{\left(a+\frac{1}{2}-b\right)}_n$ and ${(b)}_n<{(2b)}_n$ for all $n\in\{0,1,\dots\}.$ In both cases we have
$$\frac{{(b)}_n}{{(2b)}_n}\cdot\frac{{(a)}_n}{n!}\leq \frac{{\left(a+\frac{1}{2}-b\right)}_n}{{\left(b+\frac{1}{2}\right)}_n}\cdot\frac{{(a)}_n}{n!}$$
for all $n\in\{0,1,\dots\},$ and consequently for all $r\in(0,1)$ one has
\begin{equation}\label{ineq9}F(a,b;2b;r^2)\leq F\left(a,a+\frac{1}{2}-b;b+\frac{1}{2};r^2\right).\end{equation}
On the other hand, observe that if $a\geq b$ and $2a+\frac{1}{2}\geq 3b,$ then for all $n\in\{0,1,\dots\}$ we have $(a-b)n+b\left(2a+\frac{1}{2}-3b\right)\geq 0,$  and consequently the sequence $\{\gamma_n\}_{n\geq0},$ defined by
$$\gamma_n=\frac{{(b)}_n{\left(b+\frac{1}{2}\right)}_n}{{(2b)}_n{\left(a+\frac{1}{2}-b\right)}_n},$$
satisfies
$$\frac{\gamma_{n+1}}{\gamma_n}=\frac{(n+b)\left(n+b+\frac{1}{2}\right)}{(n+2b)\left(n+a+\frac{1}{2}-b\right)}\leq1$$
for all $n\in\{0,1,\dots\}.$ Thus, by using Lemma \ref{lemmapower}, the function
$$r\mapsto\frac{F(a,b;2b;r)}{F\left(a,a+\frac{1}{2}-b;b+\frac{1}{2};r\right)}$$
is decreasing on $(0,1),$ and consequently
$$\frac{F(a,b;2b;r)}{F\left(a,a+\frac{1}{2}-b;b+\frac{1}{2};r\right)}<\lim_{r\searrow0}
\frac{F(a,b;2b;r)}{F\left(a,a+\frac{1}{2}-b;b+\frac{1}{2};r\right)}=1$$
for all $r\in(0,1).$ Now, changing $r$ to $r^2$ we obtain again \eqref{ineq9}, and combining \eqref{ineq8} with \eqref{ineq9} we obtain \eqref{ineq6}. This proves part {\bf a}. The proof of part {\bf b} is similar, and thus we omit the details.
\end{proof}

We mention that by changing $r$ to $(1-r)/(1+r)$ in inequality \eqref{ineq6} we obtain for all $r\in(0,1)$ the Landen inequality
$$F\left(a,b;c;\left(\frac{1-r}{1+r}\right)^2\right)\leq\left(\frac{1+r}{2}\right)^{2a}\cdot F\left(a,b;c;1-r^2\right),$$
where $a,b$ and $c$ are as in part {\bf a} of Theorem \ref{landen2}. Moreover, if $a,b$ and $c$ are as in part {\bf b} of Theorem \ref{landen2}, then the above Landen inequality is reversed. Note that these inequalities can be obtained also by using the steps of the proof of Theorem \ref{landen2} and the formula
$$F\left(a,b;2b;1-r^2\right)=\left(\frac{1+r}{2}\right)^{-2a}\cdot F\left(a,a+\frac{1}{2}-b; b+\frac{1}{2};\left(\frac{1-r}{1+r}\right)^2\right),$$
which is the generalization of the second Landen identity in \eqref{landenid} and readily follows from \eqref{transf} by changing $r$ to $(1-r)/(1+r),$ or from \cite[p. 132]{askey}
$$F\left(a,b;2b;r^2\right)=\left(\frac{1+\sqrt{1-r^2}}{2}\right)^{-2a}\cdot F\left(a,a+\frac{1}{2}-b; b+\frac{1}{2};\left(\frac{1-\sqrt{1-r^2}}{1+\sqrt{1-r^2}}\right)^2\right)$$
by replacing $r$ with $\sqrt{1-r^2}.$

Finally, we note that for some rational values of $(a,b,c)$ the hypergeometric function $F(a,b;c;\cdot)$ reduces to some well-known special elementary functions, and thus the results of Theorems \ref{landen} and \ref{landen2} yield Landen inequalities for many elementary functions. For a list of elementary representations we refer to \cite[p. 386-387]{nist} and to the references therein. For example, if we choose the triplets $(a,b,c)=\left(\frac{1}{2},\frac{1}{2},\frac{3}{2}\right)$ and $(a,b,c)=\left(\frac{1}{2},1,\frac{3}{2}\right),$ then in view of the representations
$$F\left(\frac{1}{2},\frac{1}{2};\frac{3}{2};r^2\right)=\frac{1}{r}\arcsin r,$$
$$F\left(\frac{1}{2},1;\frac{3}{2};r^2\right)=\frac{1}{2r}\log\left(\frac{1+r}{1-r}\right),$$
and the inequalities \eqref{ineq1} and \eqref{ineq7}, we obtain the next Landen inequalities for $r\in(0,1)$
$$\frac{\sqrt{r}}{2}\arcsin\left(\frac{2\sqrt{r}}{1+r}\right)< \arcsin r,$$
$$\left(\frac{1+\sqrt{r}}{1-\sqrt{r}}\right)^{\sqrt{r}}> \frac{1+r}{1-r}.$$

\subsection*{Acknowledgment} The author wishes to acknowledge the referee's comments and suggestions which enhanced this paper.

\end{document}